\documentclass[10pt]{article}
\usepackage[english]{babel}
\usepackage[dvipsnames]{xcolor}

\usepackage{theorem,amsmath}
\usepackage{amssymb,verbatim,chicago}
\usepackage[text={6.2in,8.5in},centering]{geometry}
\usepackage[colorlinks=true, citecolor=blue]{hyperref}
\usepackage{listings}

\usepackage{epsfig}
\usepackage{enumitem}
\usepackage{extarrows}
\usepackage{graphicx}
\usepackage{float}
\usepackage{multirow}
\usepackage{caption}
\usepackage{subfigure}
\usepackage{csvsimple}
\usepackage{bm} 
\usepackage{breqn}

\theoremheaderfont{\bf}%
 \theoremstyle{change}
\newtheorem{thm}{Theorem}[section]

\newtheorem{lem}[thm]{Lemma}

{\theorembodyfont{\rmfamily}\newtheorem{rem}[thm]{Remark}}
{\theorembodyfont{\rmfamily}}
{\theorembodyfont{\rmfamily}}
{\theorembodyfont{\rmfamily}\newtheorem{ass}[thm]{Assumption}}

\newenvironment{proof}[1][\proofname]{\par \normalfont \trivlist
  \item[\hskip\labelsep\itshape
   #1]\ignorespaces
}{%
  \hspace*{\fill}$\Box$ \endtrivlist
}
\newcommand{\proofname}{Proof{\rm \,:}}

\newcommand{\ul}{\underline}
\newcommand{\ol}{\overline}
\newcommand{\bs}[1]{{\boldsymbol #1}}

\newcommand{\dd}{{\,\mathrm d}}

\renewcommand{\th}{\theta}

\newcommand{\eps}{\varepsilon}

\renewcommand{\phi}{\varphi}

\newcommand{\scr}[1]{{\mathcal #1}}

\newcommand{\EE}{\mathbb{E}}

\newcommand{\PP}{\mathbb{P}}
\newcommand{\ind}{\mathbf{1}}
\newcommand{\RR}{\mathbb{R}}

\mathchardef\given="626A

\newcommand{\bem}{\begin{bmatrix}}
\newcommand{\enm}{\end{bmatrix}}

\newcommand{\cit}{\citeNP}

\usepackage{listings}

\lstdefinelanguage{Julia}%
  {morekeywords={abstract,break,case,catch,const,continue,do,else,elseif,%
      end,export,false,for,function,immutable,import,importall,if,in,%
      macro,module,otherwise,quote,return,switch,true,try,type,typealias,%
      using,while},%
   sensitive=true,%
   morecomment=[l]\#,%
   morecomment=[n]{\#=}{=\#},%
   morestring=[s]{"}{"},%
   morestring=[m]{'}{'},%
}[keywords,comments,strings]%

\title{Bayesian nonparametric estimation in the current status continuous  mark model}
\author{Geurt Jongbloed, Frank van der Meulen and Lixue Pang\\
\\
\emph{Delft University of Technology}}
\date{\today}

\begin{document}
 \maketitle

\begin{abstract}
 In this paper we consider the current status continuous mark model where, if the event takes place before an inspection time $T$ a ``continuous mark'' variable is observed as well.
 A Bayesian nonparametric method is introduced for estimating the distribution function of the joint distribution of the event time ($X$) and mark variable ($Y$). We consider a prior that is obtained by assigning a distribution on heights of cells, where cells are obtained from a partition of the support of the density of $(X,Y)$.  As distribution on cell heights we consider both a Dirichlet prior and a prior based on the graph-Laplacian on  the specified partition. Our main result shows that under appropriate conditions, the posterior distribution function contracts pointwisely at rate  $\left(n/\log n\right)^{-\frac{\rho}{3(\rho+2)}}$, where $\rho$ is the H\"older smoothness of the true density.  In addition to our theoretical results we provide efficient computational methods for drawing from the posterior relying on a non-centred parameterisation and Crank-Nicolson updates. The performance of our computational methods is illustrated in several numerical experiments.

\end{abstract}

\section{Introduction}\label{sec:intro_3}
\subsection{Problem formulation}\label{subsec:problem-formulation}
Survival analysis is concerned with statistical modelling of the time until a particular event occurs. The event may for example be the onset of a disease or failure of  equipment. Rather than observing the time of event exactly, censoring is common in practice. If the event time is only observed when it occurs prior to a specific (censoring) time, one speaks of right censoring. In case it is only known whether the event took place before an inspection time or not, one speaks of  current status censoring. The resulting data are then called current status data.

 In this paper we consider the current status continuous mark model where, if the event takes place before an inspection time $T$, a ``continuous mark'' variable is observed as well. More specifically,
denote the event time by $X$ and the mark by $Y$.  Independent of $(X,Y)$, there is an inspection time $T$ with  density function $g$ on $[0,\infty)$. Instead of observing each $(X,Y)$ directly, we observe the inspection time $T$ together with the information whether the event occurred before time $T$ or not. If it did so, the additional mark random variable $Y$ is also observed, for which we assume  $P(Y=0)=0$. Hence, an observation of this experiment can be denoted by $W=(T, Z)=(T, \Delta\cdot Y)$ where $\Delta=\ind_{\{X\le T\}}$ (note that, equivalently,  $\Delta=\ind_{\{Z>0\}}$). This experiment is repeated $n$ times independently, leading to the observation set $\mathcal{D}_n=\{W_i,i=1,\dots,n\}$. We are interested in estimating the joint distribution function $F_0$ of $(X,Y)$ nonparametrically, based on $\mathcal{D}_n$.

 An application of this model is the HIV vaccine trial studied by \cit{HMG2007}. Here, the mark is a specifically defined viral distance that is only observed if a participant to the trial got HIV infected before the moment of inspection.

\subsection{Related literature}
In this section we review earlier research efforts on models closely related to that considered here.

Survival analysis with a continuous mark can be viewed as the continuous version of the classical competing risks model. In the latter model, failure is due to either of $K$ competing risks (with $K$ fixed) leading to a mark value that is of categorical type. As the mark variable encodes the cause of failure it is only observed if failure has occurred before inspection. These ``cause events'' are known as competing risks.
\cit{GMW2008} study nonparametric estimation for current status data with competing risks.
 In that paper, they show that the nonparametric maximum likelihood estimator (NPMLE) is consistent and converges globally and locally at rate $n^{1/3}$.

   \cit{HL1998} consider the continuous mark model under right-censoring, which is more informative compared to the current-status case because the exact event time is observed  for noncensored data. For the nonparametric maximum likelihood estimator of the joint distribution function of $(X,Y)$ at a fixed point, asymptotic normality is shown.

 \cit{HMG2007} consider  interval censoring case $k$, $k=1$ being the specific setting of current-status data considered here. In this paper the authors show that both the NPMLE and a newly introduced estimator termed ``midpoint imputation MLE''  are inconsistent. However,  coarsening the mark variable (i.e.\ making it discrete, turning the setting to that  of the competing risks model), leads to a consistent NPMLE.  This is in agreement with the results in \cit{MW2008}.

     \cit{GJW2011} and \cit{GJW2012} consider  the exact setting of this paper using frequentist estimation methods. In \cit{GJW2011}    two plug-in inverse estimators are proposed. They prove that these estimators are consistent and derive the pointwise asymptotic distribution of both estimators.
    \cit{GJW2012} define a nonparametric estimator for the distribution function at a fixed point by finding the maximiser of a smoothed version of the log-likelihood.  Pointwise consistency of the estimator is established.
	In both papers numerical illustrations are included.

\subsection{Contribution}
In this paper, we consider Bayesian nonparametric estimation of the bivariate distribution function $F_0$ in the current status continuous mark model. Approaching this problem within the Bayesian setup  has not been done before, neither from a theoretical nor computational perspective. Whereas consistent nonparametric estimators exist within frequentist statistics, convergence rates are unknown. We prove consistency and derive Bayesian contraction rates for the bivariate distribution function of $(X,Y)$ using a prior on the joint density $f$ of $(X,Y)$ that is piecewise constant. For the values on the bins we consider two different prior specifications. The first of these is defined by equipping the bin probabilities with the  Dirichlet-distribution. The other specification is close to a  logistic-normal distribution (\cit{Ait-Shen80}), where additionally smoothness on nearby bin-probabilities is enforced by taking the precision matrix of the Normal distribution equal to the graph-Laplacian induced by the grid of bins. The graph-Laplacian is a well known method to induce smoothness, see for instance \cit{Murphy} (Chapter 25.4) or \cit{Hartog} for an application in Bayesian estimation. Full details of the prior specification are in Section \ref{sec:CSCMprior}. 

 Our main result shows that under appropriate conditions, the posterior distribution function contracts pointwisely at rate  $\left(n/\log n\right)^{-\frac{\rho}{3(\rho+2)}}$, where $\rho$ is the H\"older smoothness of the true density. In this result, we assume that for the prior the bins areas where the density is constant tends to zero at an appropriate (non-adaptive) rate, as $n\to \infty$.
 The proof is based on general results from \cit{GhoVaart} for obtaining Bayesian contraction  rates. Essentially, it requires the derivation of suitable test functions and proving that the prior puts sufficient mass in a neighbourhood of the ``true'' bivariate distribution. The latter is proved by exploiting the specific structure of our prior.
 
 In addition to our theoretical results, we provide computational methods for drawing from the posterior. For the Dirichlet prior this is a simple data-augmentation scheme. For the graph-Laplacian prior we provide simple code to draw from the posterior  using probabilistic programming in the Turing Language under Julia (see \cit{Bezanson}, \cit{GeXuGha}). 
 Additionally, a much faster algorithm is introduced that more carefully exploits the structure of the problem. The main idea is to use a non-centred parameterisation (\cit{Papas}) combined with a preconditioned Crank-Nicolson (pCN) scheme (Cf.\  \cit{Cotter}).  The performance of our computational methods is illustrated in two examples. Code is available from \url{https://github.com/fmeulen/CurrentStatusContinuousMarks}.

\subsection{Outline}
The outline of this paper is as follows. In section \ref{sec:CSCM} we introduce further notation for the current status continuous mark model and detail the two priors considered. Subsequently, we present the main theorem on the 
posterior contraction rate in Section \ref{sec:CSCMposterior}. The proof of this result  is given in Section \ref{sec:lemproof_3}.
In Section \ref{sec:CSCMcomput} we present MCMC-algorithms to draw from the posterior and give numerical illustrations, including a  simulation study.

\subsection{Notation}
For two sequences $\{a_n\}$ and $\{b_n\}$ of positive real numbers, the notation $a_n\lesssim b_n$  (or $b_n\gtrsim a_n$) means that there exists a constant $C>0$, independent of $n$, such that $a_n\leq C b_n.$ We write $a_n\asymp b_n$ if both $a_n\lesssim b_n$ and $a_n\gtrsim b_n$ hold.  We denote by $F$ and $F_0$ the cumulative distribution functions corresponding to the probability densities $f$ and $f_0$ respectively.
The Hellinger distance between two densities $f,g$ is written as $h^2(f,g)=\frac1{2}\int (f^{1/2}-g^{1/2})^2$. The Kullback-Leibler divergence of $f$ and $g$ and the $L_2$-norm of $\log (f/g)$ (under $f$) by
\[
KL(f,g)=\int f\log\frac{f}{g},\qquad V(f,g)=\int f\left(\log\frac{f}{g}\right)^2
.\]

\section{Likelihood and prior specification}
\label{sec:CSCM}
\subsection{Likelihood}
\label{sec:CSCMmodel}
In this section we derive the likelihood for the joint density $f$ based on data $\scr{D}_n$. As $W_1,\ldots, W_n$ are independent and identically distributed, it suffices to derive the joint density of $W_1=(T_1, Z_1)$ (with respect to an appropriate dominating measure).
Recall that $f$ denotes the density of $(X,Y)$. Let $F$ denote the corresponding distribution function of $(X,Y)$. The marginal distribution function of $X$ is given by  \[F_X(t)=\int_0^t\int_0^\infty f(u,v)\dd v\dd u.\] Define the  measure $\mu$  on $[0,\infty)^2$ by
\begin{equation}\label{eq:mu}\mu(B)=\mu_2(B)+\mu_1(\{x\in[0,\infty): (x,0)\in B\}),\quad B\in\mathcal{B} \end{equation}
where $\mathcal{B}$ is the Borel $\sigma-$algebra on $[0,\infty)^2$ and $\mu_i$ is  Lebesgue measure on $\mathbb{R}^i$. The
 density of the law of $W_1$ with respect to $\mu$ is then given by
\begin{equation}\label{eq:density}
s_f(t,z)=g(t)\left(\ind_{\{z>0\}}\partial_2F(t,z)+\ind_{\{z=0\}}(1-F_X(t))\right),
\end{equation}
where  \[\partial_2F(t,z)=\frac{\partial}{\partial z}F(t,z)=\int_0^tf(u,z)\dd u.\] By independence of $W_1,\ldots, W_n$, 
 the likelihood of $f$ based on $\mathcal{D}_n$ is given by $l(f)=\prod_{i=1}^ns_f(T_i,Z_i)$.

\subsection{Prior specification}
\label{sec:CSCMprior}
In this section, we define a prior on the class of all bivariate density functions on $\RR^2$. Denote 
 \[\mathcal{F}=\left\{f:\RR^2\to [0,\infty) \, \colon  \int_{\RR^2}f(x,y)\dd x\dd y=1\right\}.\]
For any $f\in\mathcal{F}$, if $S$ denotes the support of $f$ and $\cup_{j}C_j, j=1,\dots, p_n$ is a partition of $S$,
we define a prior on $\mathcal{F}$ by 
\[ f_{\bs{\th}}(x,y) = \sum_{j} \frac{\th_{j}}{|C_j|} \ind_{C_{j}}(x,y),\qquad (x,y)\in \RR^2, \]
where $|C|=\mu_2(C)$ is the Lebesgue measure of the set $C$ and $\bs{\th}=(\th_1,,\dots, \th_{p_n})$. We require that all $\th_{j}$ are nonnegative and   $\sum_{j}\th_{j}=1$ (i.e.\ $\bs{\th}$ is a probability vector). 
We consider two types of prior on $\bs{\theta}$.
\begin{enumerate}
  \item {\it Dirichlet}. For a fixed parameter $\alpha=(\alpha_1,\dots,\alpha_{p_n})$ consider $\bs{\theta}\sim\mbox{Dirichlet}(\alpha)$. This prior is attractive as draws from the posterior distribution can be obtained using a straightforward data-augmentation algorithm (Cf.\ Section \ref{subsec:dirichlet}). We will refer to this prior as the {\it D-prior}.

  \item {\it Logistic-Normal with graph-Laplacian precision matrix}. For a positive-definite matrix $\Upsilon$, assume that the random vector $\bs{H}$ satisfies $\bs{H} \sim N_{p_n}(0, \tau^{-1}\Upsilon^{-1})$, for fixed positive  $\tau$. 
Next, set
\begin{equation}\label{eq:th_psiform} \th_{j} = \frac{\psi(H_{j})}{\sum_{j=1}^{p_n} \psi(H_{j})}, \quad \mbox{where}\quad\psi(x)=e^x. \end{equation} 
That is, we transform $\bs{H}$ by the softmax function implying that  realisations of $\bs{\th}$ are probability vectors. 
The matrix $\Upsilon$ is chosen as follows. The partition $\cup_j C_j$ induces a graph structure on the bins, where each bin corresponds to a node in the graph, and nodes are connected when bins are adjacent (meaning that they are either horizontal or vertical ``neighbours''). Let $L$ denote the  graph-Laplacian of the graph obtained in this way. This is the $p_n\times p_n$ matrix given by
\begin{equation}\label{eq:laplacian}  L_{i,i'} = \begin{cases} \mbox{degree node $i$} & \text{if } i=i' \\ -1 & \text{if $i\neq i'$ and nodes $i$ and $i'$ are connected} \\ 0 & \text{otherwise}.\end{cases}. \end{equation}
We take
\[ \Upsilon = L + p_n^{-2} I. \]
We will refer to this prior as the {\it LNGL-prior} (Logistic-Normal Graph Laplacian).
\end{enumerate}

\begin{rem}
Under the Dirichlet  prior values of $\th_{j}$ in adjacent bins are  negatively correlated, preventing the density to capture smoothness. This is illustrated  in the numerical study Section \ref{sec:CSCMcomput}. We would like to take a large number of bins, while not overparametrising. 
The idea of the graph-Laplacian prior is to induce positive correlation on adjacent bins and thereby specify a prior that produces smoother realisations. The numerical illustrations reveal that the posterior based on the LNGL-prior are less sensitive to the chosen number of bins in the partition, compared to the posterior based on the D-prior.

\cit{Hjort1994} (Section 2.4) has proposed a modification to the Dirichlet prior to induces smoothness among nearby bins. 
These generalised Dirichlet priors on probability vectors in $\RR^k$  take the form
\[ \pi(p_1,\ldots, p_{k-1}) \propto p_1^{\alpha_1-1} \cdot p_k^{\alpha_k - 1} \exp\left(-\tau \Delta(\bs{p})\right), \]
where $\bs{p}:=(p_1,\ldots, p_k)$ is a probability vector and $\tau>0$. Small values of $\Delta(\bs{p})$ indicates a certain characteristic is present. To connect it to the graph-Laplacian, one choice is indeed to take 
\[	 \Delta(\bs{p}) = \bs{p}^T \Upsilon \bs{p}. \]
Just as for the LNGL-prior, MCMC methods can be used to sample from the posterior in case of the generalised Dirichlet prior. The difference in the two approaches consists of imposing smoothness on the $\th_j$ directly (as in case of the generalised Dirichlet prior), or in terms of $H_j$, followed by applying the transformation in Equation \eqref{eq:th_psiform}.  We feel the approach we take in this paper is conceptually somewhat simpler. Moreover, it allows to use MCMC methods where the prior is obtained in a simple way as the pushforward of a vector of independent standard Normal random variables.  

\end{rem}

\begin{rem}
For both generalised Dirichlet prior and the LNGL-prior, the posterior mode has the usual interpretation of a penalised likelihood estimator. The choice of $\tau$ controls the amount of smoothing. In a Bayesian setting, uncertainty on $\tau$ is dealt with by employing an additional prior on $\tau$ (in the numerical section we will follow this approach). For the LNGL-prior another handle to control smoothness is be obtained by considering $\Upsilon^r$, where $r\ge 1$. \cit{Hartog} study the effect of $r$ in a simulation study in the setting of  classification on a graph and advise to take $r=1$ or $r=2$ to obtain good performance empirically. In this paper $r=1$ throughout. However, the numerical methods apply to any value $r\ge 1$. 
\end{rem}

\begin{rem}
	One can argue whether the presented prior specifications are truly nonparametric. It is not if one adopts as definition that the size of the parameter should be learned by the data. For that, a solution could be to put a prior on $p_n$ as well. While possible, this would severely complicate drawing from the posterior. As an alternative, one can take large
 values of $p_n$ (so that the model is high-dimensional), and let the data determine the amount of smoothing by incorporating flexibility in the prior. As the Dirichlet prior lacks smoothness properties, fixing large values of $p_n$ will lead to overparametrisation, resulting in high variance estimates (under smoothing). On the contrary, as we will show in the numerical examples, for the LNGL-prior, this overparametrisation can be substantially balanced/regularised by equipping the parameter $\tau$ with a prior distribution. The idea of histogram type priors with positively correlated adjacent bins has recently been used successfully in other settings as well, see for instance \cit{Gugu2018}, \cit{Gugu2019}.
\end{rem}

\begin{rem}
\cit{Ting} consider the limiting behaviour of the LNGL-prior under mesh refinement. Unsurprisingly, in the limit it behaves like the ``ordinary'' Laplace operator which is the infinitesimal generator of a driftles diffusion. 	
\end{rem}


\section{Posterior contraction}
\label{sec:CSCMposterior}
In this section we derive a contraction rate for the posterior distribution of $F_0$. Denote the posterior measure by $\Pi_n(\cdot|\mathcal{D}_n)$ (under the prior measure $\Pi_n$ described in Section \ref{sec:CSCMprior}).
\begin{ass}\label{CSCMasmp_f0}
The underlying joint density of the event time and mark,  $f_0$,   has compact support given by $\mathcal{M}=[0,M_1]\times[0,M_2]$ and is $\rho$-H\"older continuous on $\mathcal{M}$. That is, there exists a positive constant $c$  and a $\rho\in(0,1]$ such that for any $\bs{x}, \bs{y} \in \mathcal{M}$,
\begin{equation}\label{eq:f0cntns}
|f_0(\bs{x})-f_0(\bs{y})|\le c\, \|\bs{x} - \bs{y}\|^{\rho}.\end{equation}
In addition, there exist positive constants $\ul{M}$ and $\ol{M}$ such that
\begin{equation}
\label{eq:bound}
\ul{M}\le f_0(x,y)\le \ol{M},\quad\mbox{for all}\,\, (x,y)\in\mathcal{M}.
\end{equation}
\end{ass}
\begin{ass}\label{CSCMasmp_g}
The censoring density $g$ is bounded away from 0 and infinity on $(0,M_1)$. That is, there exist positive constants $\ul{K}$ and $\ol{K}$ such that  $0<\ul{K}\le g(t)\le \ol{K}<\infty$ for all $t\in(0,M_1)$ .
\end{ass}
\begin{ass}\label{CSCMasmp_prior}
For the Dirichlet prior, the parameter $\alpha=(\alpha_1,\dots,\alpha_{p_n})$ satisfies $ap_n^{-1}\le \alpha_l\le1$ for all $l=1,\dots,p_n$ and some constant $a\in \RR^+$.
\end{ass}

Let $\eps_n = (n/\log n)^{-\frac{\rho}{2(\rho+2)}}$ and $\eta_n = \eps_n^{2/3}$.  Note that $\eps_n\le \eta_n$ and $n (\eps_n^2 \wedge \eta_n^2)\to\infty$ as $n\to\infty$. Our main theoretical result is the following theorem. 
\begin{thm}\label{thm:CSCMconsistency}
Consider either of the priors defined in  Section \ref{sec:CSCMprior} and impose assumptions \ref{CSCMasmp_f0} and \ref{CSCMasmp_prior}. Fix $(x,y)\in \mathcal{M}$. Then for sufficiently large $C$
\[\EE_0\Pi_n(f\in\mathcal{F}\colon |F(x,y)-F_0(x,y)| > C\eta_n \mid \mathcal{D}_n)\to 0,\quad \mbox{as} \quad n\to\infty\]
provided that all bin areas are  equal to $\eps_n^{4/\rho}$.
\end{thm}

The proof of this theorem is based on the following two lemmas. 
\begin{lem}\label{lem:KL_3}
Fix $f_0$ and $g$ satisfying the conditions in assumption \ref{CSCMasmp_f0} and \ref{CSCMasmp_g}. If 
\begin{equation} \label{eq:Sn_3}
S_n=\left\{f\in\mathcal{F}\colon KL(s_{f_0},s_f)\le \eps_n^2,  V(s_{f_0},s_f)\le \eps_n^2\right\}.
\end{equation}
then $\Pi_n(S_n)\ge e^{-cn\eps_n^2}$ for some constant $c>0$.
\end{lem}

\begin{lem}\label{lem:uniformtest_3} Fix $(t,z)\in \scr{M}$. Define $U_n(t,z):=\{f\in\mathcal{F}\colon |F(t,z)-F_0(t,z)|>C\eta_n\}$.
There exists a sequence of test functions $\Phi_n$ such that
\begin{equation}
\label{tests_3}
\begin{split}
\EE_0(\Phi_n)&= o(1),\\
\sup_{f\in U_n(t,z)}\EE_f(1-\Phi_n)&\le c_1e^{-c_2C^2n\eps_n^2},
\end{split}
\end{equation}
for  positive constants $c_1, c_2$ and constant $C$ appearing in Theorem \ref{thm:CSCMconsistency}.
\end{lem}
The remainder of the proof of Theorem \ref{thm:CSCMconsistency} is standard and follows the general ideas in  \cit{Ghosal}.

%
%

\section{Proof of Lemmas}\label{sec:lemproof_3}

\subsection{Proof of lemma \ref{lem:KL_3}}
\begin{proof}
The proof consists of 4 steps:
\begin{enumerate}
	\item constructing a subset $\Omega_n$ of $S_n$ for which the prior probability can be bounded from below;
	\item deriving this bound for the D-prior;
	\item deriving this bound for the LNGL-prior;
	\item verifying that the prior mass condition is satisfied for the choice of $\eps_n$ in the lemma. 
\end{enumerate}
In the proof we will use double-indexation of coefficients, so rather than $\theta_j$ we write $\theta_{j,k}$, where $(j,k)$ indexes a particular bin. A similar convention applies to $H$. 

{\bf Step 1.} We first give a sequence of approximations for $f_0$. Let $\delta_n$ be a sequence of positive numbers tending to $0$ as $n\to \infty$. 
For each $n$, let $A_{n,j}$, $B_{n,k}$ be sets such that $\cup_{j=1}^{J_n} \cup_{k=1}^{K_n}(A_{n,j}\times B_{n,k})$ is a  partition of $\mathcal{M}$, where the integers $J_n$ and $K_n$ are  chosen such that $|A_{n,j}|=|B_{n,k}|=\delta_n$ for all $j$ and $k$.

Let $f_{0,n}$ be the piecewise constant density function defined by
\begin{equation}\label{eq:f0n_def}f_{0,n}(t,z)=\sum_{j=1}^{J_n}\sum_{k=1}^{K_n} \frac{w_{0,j,k}}{|A_{n,j}\times B_{n,k}|}\ind_{A_{n,j}\times B_{n,k}}(t,z),\end{equation}
where $w_{0,j,k}=\int_{A_{n,j}}\int_{B_{n,k}}f_0(u,v)\dd v\dd u$. That is, we approximate $f_0$ by averaging it on each bin. 
Set $\eps_n^2 = \delta_n^\rho$. By Lemma \ref{lem:subset_omega} in the appendix there exists a constant $C>0$ such that the set defined by 
\begin{equation}
\label{eq:omega}
\bar\Omega_n:=\left\{f\in\mathcal{F}: ||f-f_{0,n}||_\infty\le C \delta_n^{\rho},\, \mbox{supp}(f)\supseteq\mathcal{M}\right\}.
\end{equation}
satisfies $\bar\Omega_n\subseteq S_n$.  

Recall that $p_n=J_nK_n$ denote the total number of bins. According to the prior specifications in Section  \ref{sec:CSCMprior}, for any $f\in\scr{F}$, we parameterize
\[ f_{\bs{\th}}(x,y) = \sum_{j,k} \frac{\th_{j,k}}{|A_{n,j} \times B_{n,k}|} \ind_{A_{n,j} \times B_{n,k}}(x,y),\qquad (x,y)\in \RR^2, \]
where $\bs{\th}$ denotes the vector obtained by stacking all coefficients $\{\th_{j,k},j=1,\dots,J_n,k=1,\dots,K_n\}$. 
For any $(t,z)\in A_{n,j}\times B_{n,k}$, $j,k\ge 1$, we have
\[|f_{\bs{\th}}(t,z)-f_{0,n}(t,z)|=|A_{n,j} \times B_{n,k}|^{-1}|\th_{j,k}-w_{0,j,k}|\le \delta_n^{-2}\max_{j,k}|\th_{j,k}-w_{0,j,k}|.\]
 Hence
\begin{equation}\label{eq:subsettheta} \Omega_n:=\left\{f_{\bs{\th}}\in\mathcal{F}\,\colon \,\max_{j,k}\, |\th_{j,k}-w_{0,j,k}|\le C\delta_n^{\rho+2}\right\}\subseteq\bar\Omega_n\end{equation}
and consequently $\Pi_n(S_n) \ge \Pi_n(\Omega_n)$. 

{\bf Step 2.}
For the D-prior we have $\bs{\theta}\sim\mbox{Dirichlet}(\alpha)$ for fixed $\alpha=(\alpha_1,\dots,\alpha_{p_n})$, where we assume that  $ap_n^{-1}\le\alpha_l\le 1$ for all $l=1,\dots,p_n$.
By Lemma 6.1 in \cit{Ghosal}, we have
\begin{align*}
\Pi_n(\Omega_n)
&\ge \Gamma\Bigg(\sum_{l=1}^{p_n}\alpha_l\Bigg)\left(C\delta_n^{\rho+2}\right)^{p_n}\prod_{l=1}^{p_n}\alpha_l\\
&\ge \exp\left(\log \Gamma(a)+p_n\log\left(C\delta_n^{\rho+2}\right)+p_n\log (ap_n^{-1})\right)
\end{align*}
As $p_n \asymp \delta_n^{-2}$, $p_n \log \delta_n^{\rho+2} \asymp p_n \log p_n^{-(\rho+2)/2}= -(\rho/2+1) p_n \log p_n$. We conclude that the exponent behaves asymptotically as a multiple of $-p_n \log p_n$. 

{\bf Step 3.} Let $\th_{j,k} = \frac{\psi(H_{j,k})}{\sum_{j,k} \psi(H_{j,k})}$ as defined in (\ref{eq:th_psiform}).
For the LNGL-prior, as will become clear shortly, we need a lower bound on $|\Upsilon|$ and upper bound on the largest eigenvalue of $\Upsilon$. For the latter,  note that  all eigenvalues of any stochastic matrix are bounded by $1$. For the graph considered here, there exist diagonal matrices $U$ and $D$ such that $D(U + L)$ is a stochastic matrix (simply adding elements to the diagonal to have a matrix with nonnegative elements, followed by renormalising each row to have its elements sum to 
one. This procedure implies that in the present setting the largest eigenvalue of $L$ is bounded by $8$. Therefore, the largest eigenvalue of $\Upsilon$ is bounded by $9$. As $L$ has smallest eigenvalue $0$,  the smallest eigenvalue of $\Upsilon$ equals $p_n^{-2}$. Hence \[|\Upsilon| \ge (p_n^{-2})^{p_n} = \exp\left(-2p_n \log p_n\right). \]

We have the following bounds:
\begin{align*}
\delta_n^2 \| f- f_{0,n}\|_\infty &\le  \sup_{j,k}\, |\theta_{j,k}- \theta_{0,j,k}| = \| S(\bs{H}) - S(\bs{H}_0)\|_\infty \\ & \le
 \| S(\bs{H}) - S(\bs{H}_0)\|_2 \le \|\bs{H}-\bs{H}_0\|_2, 
\end{align*}
assuming $S$ is the softmax function. At the last inequality we use that this function is Lipschitz-continuous with respect to $\|\cdot\|_2$ (Cf.\  Proposition 4 in \cit{GaoPavel}). 
Therefore 
\[ \underline{\Omega}_n := \{ \bs{H}\,:\, \|\bs{H}-\bs{H}_0\|_2 \le C \delta_n^{2} \Delta_n\} \subseteq \{f\,:\, \|f-f_{0,n}\|_\infty \le C\Delta_n\},. \]
where $\Delta_n := \delta_n^\rho$. Whereas  notationally implicit, both sets in the preceding display are sets of $\omega$s induced by the constraints on either $\bs{H}$ or $f$. 
Hence it suffices to lower bound 
\[ \int  \int_{\underline{\Omega}_n} \phi(\bs{H}; 0, \tau \Upsilon^{-1}) \dd \bs{H} f(\tau) \dd \tau .
\]

We first focus on the inner integral. As the largest eigenvalue of $\Upsilon$ s bounded by $9$, we have $\bs{H}^T \Upsilon \bs{H} \le 9 \|\bs{H}\|_2^2$.
Hence, 
\begin{align*}
\scr{I}_n(\tau)&:= \int_{\underline{\Omega}_n} \phi(\bs{H}; 0, \tau \Upsilon^{-1}) \dd \bs{H} \\& = (2\pi \tau)^{-p_n/2} |\Upsilon|^{1/2}	 \int_{\underline{\Omega}_n} \exp\left(-\frac1{2\tau} \bs{H}^T \Upsilon \bs{H} \right) \dd \bs{H} \\ & \ge  (2\pi \tau)^{-p_n/2}\exp\left( -p_n \log p_n\right) \int_{\underline{\Omega}_n}\exp\left(-\frac1{2\tau} 9 \|\bs{H}\|^2 \right) \dd \bs{H}
 \end{align*}
 If $\bs{H} \in \underline{\Omega}_n$, then 
 \[ \|\bs{H}\|^2 \le \|\bs{H}-\bs{H}_0\|^2 + \|\bs{H}_0\|^2 \le C^2\delta_n^{4} \Delta_n^2 +  \|\bs{H}_0\|^2. \]
which means that 
  \begin{align*}
\scr{I}_n(\tau)&\ge (2\pi \tau)^{-p_n/2}\exp\left( -p_n \log p_n\right) \int_{\underline{\Omega}_n}\exp\left(-\frac1{2\tau} 9\left(C^2\delta_n^{4}\Delta_n^2+\|\bs{H}_0\|^2\right) \right) \dd \bs{H}\\ &= (2\pi \tau)^{-p_n/2}\exp\left( -p_n \log p_n\right) \exp\left(-\frac1{2\tau} 9 \left(C^2\delta_n^{4}\Delta_n^2+\|\bs{H}_0\|^2\right) \right) \mbox{Vol}(\underline\Omega_n).
\end{align*}
Hence, 
\begin{equation}\label{eq:I_n(tau)} \log \scr{I}_n(\tau) \ge -\frac{p_n}{2} \log (2\pi \tau) - p_n \log p_n - \frac1{2\tau} 9\left(C^2\delta_n^{4}\Delta_n^2+\|\bs{H}_0\|^2 \right) + \log \mbox{Vol}(\bar\Omega_n). \end{equation}
We express the asymptotic behaviour of all terms in the exponent in terms of $p_n$. To this end, note that $p_n \asymp \delta_n^{-2}$, hence $\Delta_n \asymp p_n^{-\rho/2}$, leading to $ \delta_n^4 \Delta_n^4 \asymp p_n^{-2} p_n^{-\rho} = p_n^{-\rho-2} $. As $\mbox{Vol}(\bar\Omega_n) \asymp \Delta_n^{p_n} $ we have $\log \mbox{Vol}(\bar\Omega_n) \asymp p_n \log \Delta_n \asymp -\frac{\rho}{2} p_n \log p_n$. 
 So we can conclude that the right-hand-side of \eqref{eq:I_n(tau)} behaves as $-p_n \log p_n$ for $n$ large.

 {\bf Step 4.} For both priors, the prior mass condition gives the following condition on $\eps_n$:
 \[ p_n \log p_n \lesssim n \eps_n^2. \]
As $\delta_n^\rho = \eps_n^2$ we get $p_n \asymp  \delta_n^{-2} = \left(\eps_n^{2/\rho}\right)^{-2}=\eps_n^{-4/\rho}$.
Hence we need to choose $\eps_n$ such that 
\[  \eps_n^{-4/\rho} \log (1/\eps_n) \lesssim n\eps_n^2. \]
This relationship is satisfied if  
\[ \eps_n \asymp  (n/\log n)^{-\frac{\rho}{4+2\rho}}. \]

\end{proof}


\subsection{Proof of lemma \ref{lem:uniformtest_3}}
\begin{proof}
We consider different test functions in different regimes of $t$: $t\in (0,M_1)$ and $t\in\{0,M_1\}$.

First consider $(t,z)\in (0,M_1)\times(0,M_2]$. Define test sequences
\begin{align*}
\Phi_n^+(t,z)=\ind\left\{\frac1{n}\sum_{i=1}^n \kappa_n^+(t,z;T_i,Z_i)-\int_t^{t+h_n}g(x)F_0(x,z)\dd x>e_n/2\right\},\\
\Phi_n^-(t,z)=\ind\left\{\frac1{n}\sum_{i=1}^n \kappa_n^-(t,z;T_i,Z_i)-\int_{t-h_n}^tg(x)F_0(x,z)\dd x<-e_n/2\right\},
\end{align*}
where \begin{align*}
\kappa_n^+(t,z;T,Z)&=\ind_{[t,t+h_n]}(T)\ind_{(0,z]}(Z),\\
\kappa_n^-(t,z;T,Z)&=\ind_{[t-h_n,t]}(T)\ind_{(0,z]}(Z).\end{align*}
Note that if $(T,Z) \sim f$, then 
\begin{align*}
\EE_f(\kappa_n^+(t,z;T,Z))&=\int \ind_{[t,t+h_n]}(x)\ind_{(0,z]}(u) s_f(x,u)\dd \mu(x,u)\\
&= \int_t^{t+h_n}\int_0^z g(x)\partial_2F(x,u)\dd \mu_2(x,u)
=\int_t^{t+h_n}g(x)F(x,z)\dd x,
\end{align*}
where $s_f$ is the density function of $(T,Z)$ defined in (\ref{eq:density}). 
By  Assumption \ref{CSCMasmp_g},
\[  \EE_f \left[(\kappa_n^+(t,z;T,Z))^2\right] = \EE_f(\kappa_n^+(t,z;T,Z))   
\le \int_t^{t+h_n}g(x)\dd x\le \ol{K}h_n.  \] 

 The same upper bound holds for $\EE_f \left[(\kappa_n^-(t,z;T,Z))^2\right]$.
By Bernstein's inequality (\cit{AWVaart}, Lemma 19.32),
\[\EE_0(\max(\Phi_n^+(t,z),\Phi_n^-(t,z)))\le 2\exp\left(-\frac1{16}\frac{ne_n^2}{\ol{K}h_n+e_n/2}\right)=o(1).\]

%

Let $\{\eta_n\}$ be a sequence of positive numbers tending to zero as $n\to \infty$. 
For $(t,z)\in[0,M_1]\times(0,M_2]$, define sets
\begin{align*}
U_{n,1}(t,z)&=\{f: F(t,z)>F_0(t,z)+C\eta_n\},\\
U_{n,2}(t,z)&=\{f: F(t,z)<F_0(t,z)-C\eta_n\}.\end{align*}
Then $U_n(t,z)=U_{n,1}(t,z)\cup U_{n,2}(t,z)$.

When $f\in U_{n,1}(t,z)$, for any $x\in[t,t+h_n]$, by the monotonicity of $F$ and $f_0\le \ol{M}$, we have
\begin{align*}F(x,z)-F_0(x,z)&\ge F(t,z)-F_0(t,z)-(F_0(x,z)-F_0(t,z))\\
&\ge C\eta_n-\ol{M}M_2h_n\ge C\eta_n/2.
\end{align*}
Then it follows
\[\int_t^{t+h_n}g(x)(F(x,z)-F_0(x,z))\dd x\ge \frac{C\eta_n}{2}\int_t^{t+h_n}g(x)\dd x\ge \frac{C\ul{K}}{2}\eta_n h_n=e_n.\]
Hence, for $f\in U_{n,1}$ we have
\begin{align*}\EE_f(1-\Phi_n^+(t,z))&=\PP_f\left(\frac1{n}\sum_{i=1}^n \kappa_n^+(t,z|T_i,Z_i)-\int_t^{t+h_n}g(x)F_0(x,z)\dd x<e_n/2\right)\\
&\le \PP_f\left(\frac1{n}\sum_{i=1}^n \kappa_n^+(t,z|T_i,Z_i)-\int_t^{t+h_n}g(x)F(x,z)\dd x\le -e_n/2\right).
\end{align*}
Further, Bernstein's inequality gives
\[\EE_f(1-\Phi_n^+(t,z))\le 2\exp\left(-\frac1{16}\frac{n e_n^2}{\ol{K}h_n+e_n/2}\right)\le c_1 e^{-c_2C^2n\eps_n^2}\]
for some constants $c_1, c_2>0$.

When $f\in U_{n,2}(t,z)$, $x\in[t-h_n,t]$, we have
\begin{align*}F(x,z)-F_0(x,z)&\le F(t,z)-F_0(t,z)+F_0(t,z)-F_0(x,z)\\
&\le -C\eta_n+\ol{M}M_2h_n\le -C\eta_n/2
\end{align*}
and
\[\int_{t-h_n}^tg(x)(F(x,z)-F_0(x,z))\dd x\le -\frac{C\ul{K}}{2}\eta_n h_n=-e_n.\]
Hence for $f\in U_{n,2}$, the  type II error satisfies
\[\EE_f(1-\Phi_n^-(t,z))
\le \PP_f\left(\frac1{n}\sum_{i=1}^n \kappa_n^-(t,z|T_i,Z_i)-\int_{t-h_n}^tg(x)F(x,z)\dd x\ge e_n/2\right).
\]
Using Bernstein's inequality again, we have
\[\EE_f(1-\Phi_n^-(t,z))\le c_1 e^{-c_2C^2n\eps_n^2},\quad \mbox{for some}\,\,c_1,c_2>0.\]

\medskip
For the boundary case $(t,z)\in \{0,M_1\}\times(0,M_2]$. With a similar idea, in order to give non-zero test sequences, we use $\kappa_n^+$ define $\Phi^+_n(0,z), \Phi^-_n(0,z)$ and $\kappa_n^-$ define $\Phi^+_n(M_1,z), \Phi^-_n(M_1,z)$.
When $f\in U_{n,1}(0,z)$, using the tests sequence $\Phi_n^+(0,z)$ defined in case $t\in(0,M_1)$, we have
\[
\sup_{f\in U_{n,1}(0,z)}\EE_f(1-\Phi_n^+(0,z))\le c_1 e^{-c_2C^2n\eps_n^2}.\]
When $f\in U_{n,2}(M_1,z)$, using the tests sequence $\Phi_n^-(M_1,z)$ defined in case $t\in(0,M_1)$, we have
\[\sup_{f\in U_{n,2}(M_1,z)}\EE_f(1-\Phi_n^-(M_1,z))\le c_1 e^{-c_2C^2n\eps_n^2}.\]

Note that for any $f\sim\Pi_n$ and $t\in A_{n,j}$, $j=1,\dots,J_n$,
\begin{equation}\label{eq:fbound}
\int_0^{M_2}f(t,v)\dd v=|A_{n,j}|^{-1}\sum_{k=1}^{K_n}\th_{j,k}\le \delta_n^{-1}K_n=M_2.\end{equation}
Here we use $\th_{j,k}\le 1$ and $|A_{n,j}|\ge \delta_n$.
When $f\in U_{n,2}(0,z)$, for any $x\in[0,h_n]$, using (\ref{eq:fbound}) we have
\begin{align*}
F(x,z)-F_0(x,z)&\le F(x,z)-F(0,z)+F(0,z)-F_0(0,z)\\
&\le \int_0^x\int_0^zf(u,v)\dd v\dd u-C\eta_n\\
&\le M_2h_n-C\eta_n\le-C\eta_n/2
\end{align*}
and
\[\int_0^{h_n}g(x)(F(x,z)-F_0(x,z))\dd x\le -e_n.\]
Define 
tests sequence
\[\Phi_n^-(0,z)=\ind\left\{\frac1{n}\sum_{i=1}^n \kappa_n^+(0,z|T_i,Z_i)-\int_0^{h_n}g(x)F_0(x,z)\dd x<-e_n/2\right\}.\]
Hence by the Bernstein's inequality,
\[
\EE_f(1-\Phi_n^-(0,z))\le c_1e^{-c_2C^2n\eps_n^2}.\]

By the similar arguments as above, when $f\in U_{n,1}(M_1,z)$, for any $x\in[M_1-h_n,M_1]$, using (\ref{eq:fbound}) we have
\begin{align*}
F(x,z)-F_0(x,z)&\ge F(x,z)-F(M_1,z)+F(M_1,z)-F_0(M_1,z)\\
&\ge C\eta_n-\int_{M_1-h_n}^{M_1}\int_0^zf(u,v)\dd v\dd u\\
&\ge C\eta_n-M_2h_n\ge C\eta_n/2
\end{align*}
and
\[\int_0^{h_n}g(x)(F(x,z)-F_0(x,z))\dd x\ge -e_n.\]
Define tests sequence
\[\Phi_n^+(M_1,z)=\ind\left\{\frac1{n}\sum_{i=1}^n \kappa_n^-(M_1,z|T_i,Z_i)-\int_{M_1-h_n}^{M_1}g(x)F_0(x,z)\dd x>e_n/2\right\},\]
hence,
\[
\EE_f(1-\Phi_n^+(M_1,z))\le c_1e^{-c_2C^2n\eps_n^2}.\]

To conclude, take $\Phi_n(t,z)=\max(\Phi_n^+(t,z),\Phi_n^-(t,z))$, we derived
\begin{align*}
\EE_0\Phi_n(t,z)&=o(1),\\
\sup_{f\in U_n(t,z)}\EE_f(1-\Phi_n(t,z))&\le c_1 e^{-c_2C^2n\eps^2_n}.\end{align*}
\end{proof}

\section{Computational methods}\label{sec:CSCMcomput}
In this section we present algorithms for drawing from the posterior distribution for both priors described in  section \ref{sec:CSCMprior}.  Contrary to our theoretical contribution, we include a
 prior on the scaling parameter $\tau$  appearing in the LNGL prior. 
 Likewise, for the D-prior we will assume
 \begin{align*}
	\tau & \sim \Pi\\
	\th \mid \tau &\sim \mbox{Dirichlet}(\tau,\ldots, \tau). 
\end{align*}
 to ensure a fair comparison of simulation results obtained by either of these priors.

\subsection{Dirichlet prior (D-prior)}\label{subsec:dirichlet}
First, we consider the case where
$\{(X_i,Y_i),\, i=1,\ldots,n\}$ is a sequence of independent random vectors, with common density $f_0$ that is piecewise constant on $A_{n,j}\times B_{n,k}$ and compactly supported. This ``no-censoring'' model has likelihood
\[ l(\bs{\theta}) = \prod_{j,k} \th_{j,k}^{C_{j,k}}, \]
where $C_{j,k}=\sum_i \ind\{(X_i, Y_i) \in A_{n,j} \times B_{n,k}\}$ denotes  the number of observations that fall in bin $A_{n,j} \times B_{n,k}$. Clearly, the Dirichlet prior is conjugate for the likelihood,  resulting in the  posterior being of Dirichlet type as well and known in closed form.
In case of censoring, draws from the posterior for the Dirichlet prior can be obtained by data-augmentation, where the following two steps are alternated:
\begin{enumerate}
  \item Given $\bs{\theta}$ and censored data, simulate the ``full data''. This is tractable since the censoring scheme tells us in which collection of  bins the actual observation can be located. Then one can renormalise the density $f$ restricted to these bins and select a specific bin accordingly and generate the ``full data''. Cf.\ Figure \ref{fig:censoring-sketch} for the two types of observations.
  \item Given the ``full data'', draw samples for $\bs{\theta}$ from the posterior which is of Dirichlet type.
\end{enumerate}
\begin{figure}
\begin{center}
	\includegraphics[scale=0.4]{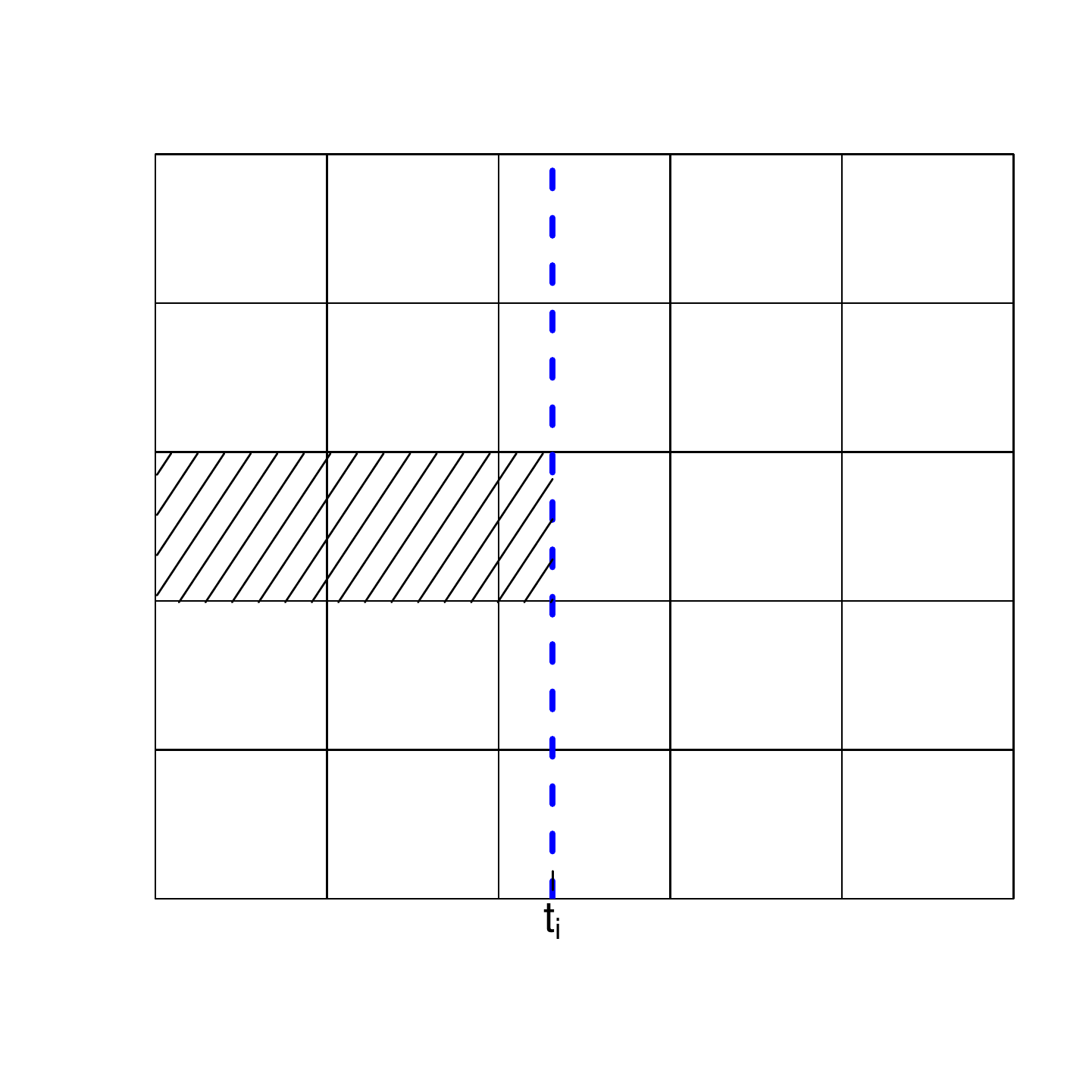}
    \includegraphics[scale=0.4]{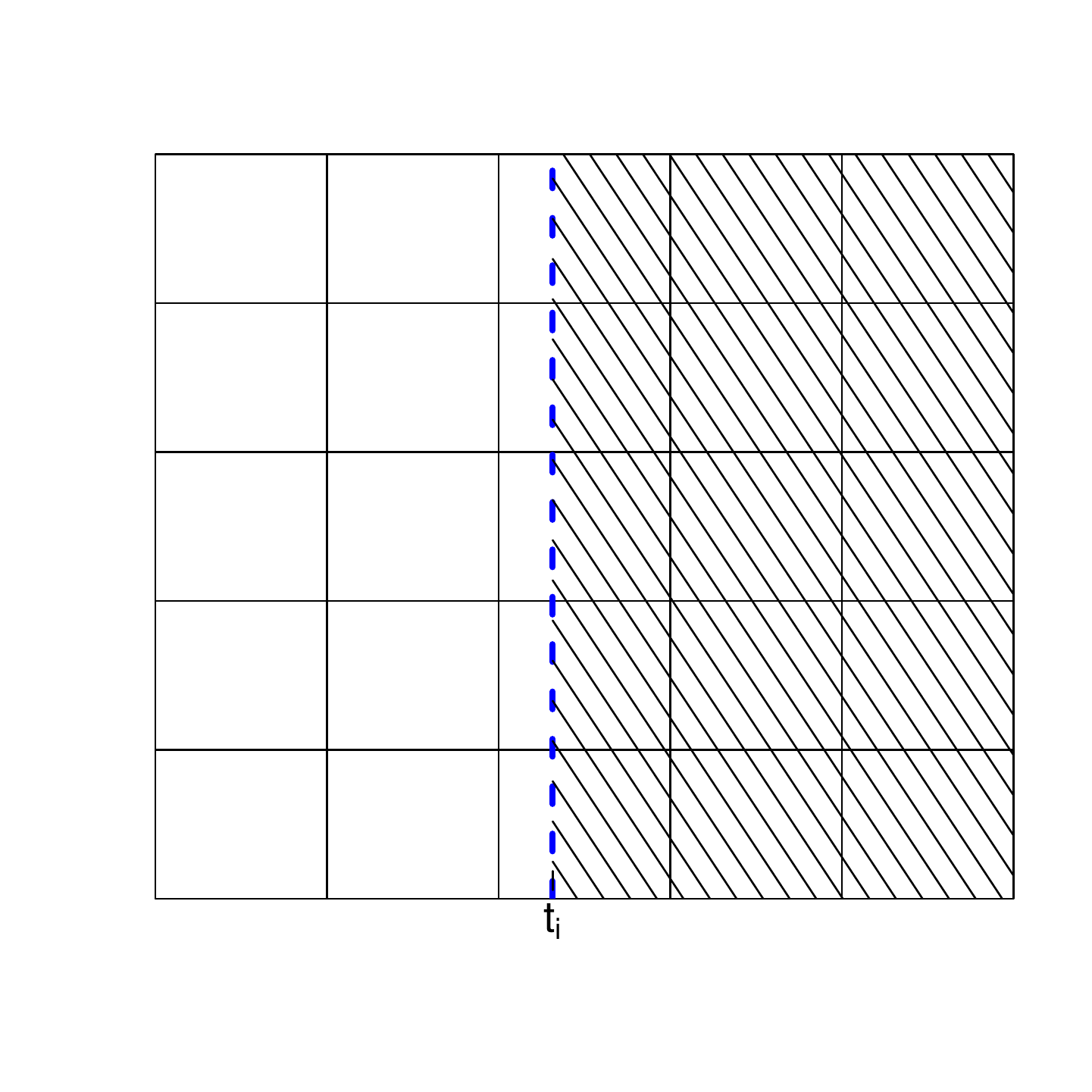}
	\caption{Left: if $x_i \le t_i$ the mark is observed. Right: if $x_i> t_i$ the mark is not observed. \label{fig:censoring-sketch}}
\end{center}
\end{figure}

\subsection{Logistic Normal Graph Laplacian prior (LNGL-prior)}

For the LNGL prior, one could opt for a data-augmentation scheme as well, but its attractiveness  is lost since step (2) above is no longer of  simple form. Therefore, we propose to bypass data-augmentation in this case. In an initial version of this paper we have proposed to use the probabilistic programming language {\sffamily Turing} (see \cit{GeXuGha}) that is based on the {\sffamily Julia} language 
(see \cit{Bezanson}). With modest programming efforts, samples from the posterior can then be obtained by Hamiltonian Monte Carlo methods. For completeness, we present the core of such code in Appendix \ref{sec:turing}. As can be seen, this contains about 10 lines of code and reads intuitively. However, a much faster algorithm can  be derived by exploiting structure in the statistical model. The main idea is to use a non-centred parameterisation (\cit{Papas}) combined with a preconditioned Crank-Nicolson (pCN) scheme. The latter scheme dates back to \cit{Neal98}; a more recent exposition can be found in \cit{Cotter} (see also \cit{vdMSchauer} in a different application setting). 

Let $\theta$ denote the parameter vector (with all $\theta_{j,k}$ stacked). Let $S: \RR^p \to \RR^p$ denote the ``softmax''-function, defined by $S(x_1,\ldots, x_p) = (e^{x_1}, \ldots, x^{x_p})/\sum_{i=1}^p e^{x_i}$ (this is a convenient choice, for our algorithm other mappings of  a vector to a probability vector are also allowed). To sample from the prior of  $(\tau, \theta)$, with $\theta\in\RR^p$, we sample according to the following scheme
\begin{align*}
		\tau & \sim \Pi \\
		 z &\sim  \mbox{N}_p(0,I_p) \\
	\theta \mid z, \tau &= S(\bar{U} z\sqrt{\tau}),	 
\end{align*}
where  $\bar U = U^{-1}$ with $U$  obtained from the Cholesky-decompostion of the graph-Laplacian matrix $\Upsilon$ (with a small multiple of the identity matrix added):  $\Upsilon  = U^T U$.  Hence, we have introduced the random vector $z$, which is centred in between $\tau$ and $\theta$. The algorithm we propose is a Gibbs sampler which iteratively updates $\tau$ and $z$. 
While our theoretical contribution assumes a prior on $\tau$ of Gamma type (see Assumption \ref{CSCMasmp_prior}), the MCMC-algorithms in this section apply more generally to a prior distribution $\Pi$ on $\tau$ that is supported on the positive halfline. Small values of $\tau$ induce more smoothing.

For updating $z$ conditional on $\tau$ we use the pCN scheme: pick a tuning parameters $\rho\in[0,1)$ (typically chosen close to $1$) and propose a new  value $z^\circ$ for $z$ by setting
\begin{equation}\label{eq:zcirc} z^\circ = \rho z + \sqrt{1-\rho^2} w, \end{equation}
where $w\sim \mbox{N}_p(0,I_p)$, independently of $(\tau, z)$. Next, the Metropolis-Hastings (MH) acceptance rule is used to accept the proposal  $z^\circ$ with probability $1 \wedge \scr{L}(z^\circ, \tau)/ \scr{L}(z, \tau)$, where $\scr{L}(z,\tau)$ denotes the likelihood, evaluated in $(z,\tau)$ (note that the prior ratio and proposal ratio cancel, as $\pi(z) q(z^\circ \mid z)$ is symmetric in $(z,z^\circ)$).  The likelihood can be simply and efficiently computed. To see this, consider Figure \ref{fig:censoring-sketch}. Observations can be represented either by the left- or right figure. For the $i$-th observation, we will denote  by $a_i$ the vector that contains for each cell of the partition the fraction of the area that is shaded (so most values will be either $0$ or $1$, only cells intersected by the blue dashed line will have entries in $(0,1)$). 
 With this notation, it is easy to see that 
\[ \log \scr{L}(z, \tau) = \sum_{i=1}^n \log (\theta^T  a_i), \quad \text{with} \quad  \theta =  S(\bar{U} z\sqrt{\tau}).\]
To efficiently compute the loglikelihood, all that needs to be computed (once), is for each observation   the vector of indices corresponding to the shaded boxes and the corresponding area fractions (this enables to exploit sparsity in $a_i$ when computing $\theta^T a_i$). Also note that the Cholesky decomposition of $L$ only needs to be computed once.

For updating $\tau$, conditional on the data and $z$, we use a MH-step.
We draw a proposal $\tau^\circ$ according to  $\log \tau^\circ \mid \tau  \sim \mbox{N}(\log \tau, \delta^2)$. It is accepted with probability 
\[ 1 \wedge \frac{\scr{L}(z, \tau^\circ)}{\scr{L}(z, \tau)} \frac{\pi(\tau^\circ)}{\pi(\tau)} \frac{\tau^\circ}{\tau}, \]
where $\pi$ denotes the prior density on $\tau$ and the term $\tau^\circ/\tau$ comes from the Jacobian. 
Note that a partial conjugacy for updating $\tau$ gets lost, even when employing a prior of inverseGamma type. 

We conclude that  in one iteration of the Gibbs-sampler, it is only required to do a simple MH-step for updating $\tau$,   sample $z^\circ$ as in \eqref{eq:zcirc} and use the MH-acceptance rule for this step as well. 


\begin{rem}
In the prior specification we have postulated $\theta \mid z, \tau = S(\bar{U} z\sqrt{\tau})$. Note that the prior on $z$ centers at zero. However, due to translation invariance of the softmax function, the same prior on $\theta$ is specified is we take 	
$\theta \mid z, \tau = S(\bar{U} z\sqrt{\tau} + v)$, where $v$ is a vector that is a multiple of the vector with all elements equal to $1$. Despite this identifiability issue, we have not encountered severe autocorrelation in traceplots for the LNGL-prior. 
\end{rem}

\begin{rem}
The same algorithm applies to more general censoring schemes. For example, suppose there are two checkup times, and the mark is only observed if the event took place in between those checkup times. This setting just corresponds to a different type of shading of boxes in Figure \ref{fig:censoring-sketch} but otherwise does not implicate any change to the numerical setting. Hence, the interval-censored continuous-mark model (see for instance \cit{MW2008}) is covered by our computational methods. 
\end{rem}

\subsection{Numerical examples}

In the following simulations, we compare the priors based on the Dirichlet distribution and the LNGL-prior.
Updating $\tau$ can be done by incorporating  a MH-step similar as for the LNGL-prior.

In each of the reported results, $20,000$ MCMC iterations were used, posterior means were computing after discarding the initial $1/3$ of the iterations as burnin. The algorithms were tuned such that the MH-steps have acceptance probability of about $0.25-0.5$. For analysing a data set of size $n=200$ with $100$ bins,  computing times are about $3$ seconds for $20,000$ iterations, using  a Macbook-pro 2 GHz Quad-Core Intel Core i5 with 16 GB RAM. Note that the complexity of the algorithm scales with the number of bins and not with the sample size. From experiments with a large number of bins it appears that the acceptance rate for the pCN-step does not deteriorate, as also seen in other settings where pCN is utilised (Cf.\ \cite{Cotter}). Julia code (\cit{Bezanson}) is available from \url{https://github.com/fmeulen/CurrentStatusContinuousMarks}. Computed Wasserstein distances are computed using the {\sffamily transport} library in {\sffamily R}.

Traceplots (not included here) confirm that in our experiments the chain on $(\theta, \tau)$ mixes well. Only in case of the D-prior, somewhat larger autocorrelation is observed in the chain for $\tau$.  The prior measure $\Pi$ is taken to be the standard Exponential distribution for both priors. 

We will consider the following data generated as independent replications of $(X,Y)$, where $(X,Y)$ is drawn from the mixture experiment
  \[ (X,Y) \sim \begin{cases} (U,V) & \text{with probability $0.3$}\\
  (1-U,V) & \text{with probability $0.7$} 			
 			\end{cases},
 \]	
with $(U,V)$ having density
\[	f(u,v) = (3/8)(u^2+v)\ind_{[0,1]\times [0,2]}(u,v). \]
In this case, it is easy to sample $(X,Y)$ by first sampling from the marginal distribution of $X$ and then from the density of $Y\mid X$. In both cases the inverse of the cumulative distribution function can be computed in closed form. 
We assume that the censoring random vairable $T \sim \sqrt{U}$ where $U$ is uniformly distributed on $[0,1]$. This implies that the density of $T$ is given by $t\mapsto 2t\ind_{[0,1]}(t)$. For the LNGL-prior we take
$\Upsilon = L + N^{-2}I$ where $L$ is defined in (\ref{eq:laplacian}) and $N$ is the number of cells in the partition used to cover the support of the density. 


\subsubsection{Results for one dataset}
We sample $200$ observations from the model. The true density with observations superimposed is depicted in Figure \ref{fig:obsB}. Horizontally/vertically we took $25/50$ bins, yielding equally sized bins. Note that the number of bins is way larger than the sample size; smoothing/penalisation being enforced by the prior. 
\begin{figure}
\begin{center}
	\includegraphics[scale=0.8]{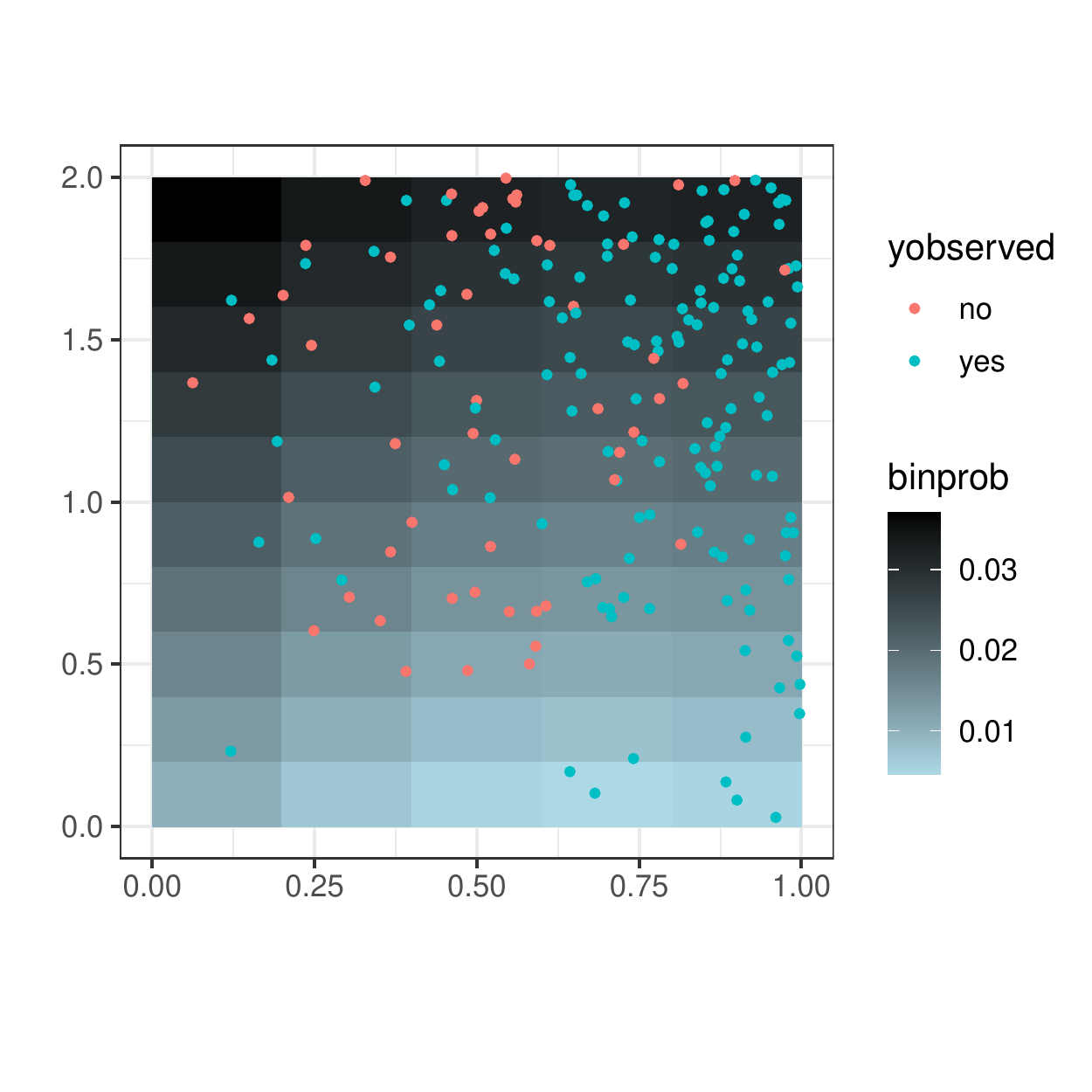}
\end{center}
\caption{For each cell in the partition of $[0,1]\times [0,2]$, the probability of the cell is coloured. Superimposed are points representing the data, where the coordinate along the horizontal axis is the censoring time, and the coordinate along the vertical axis is the mark-variable $y$. If for a particular point $y$ is observed/unobserved (cyan/red colour), then this means the event time is to the left/right of the censoring time $t$. The height of the red points is latent in the observations. \label{fig:obsB}}
\end{figure}
In Figure  \ref{fig:resB_postmean} we compare the performance of the posterior mean estimator under both prior specifications. 
\begin{figure}
	\begin{center}
		\includegraphics[scale=0.8]{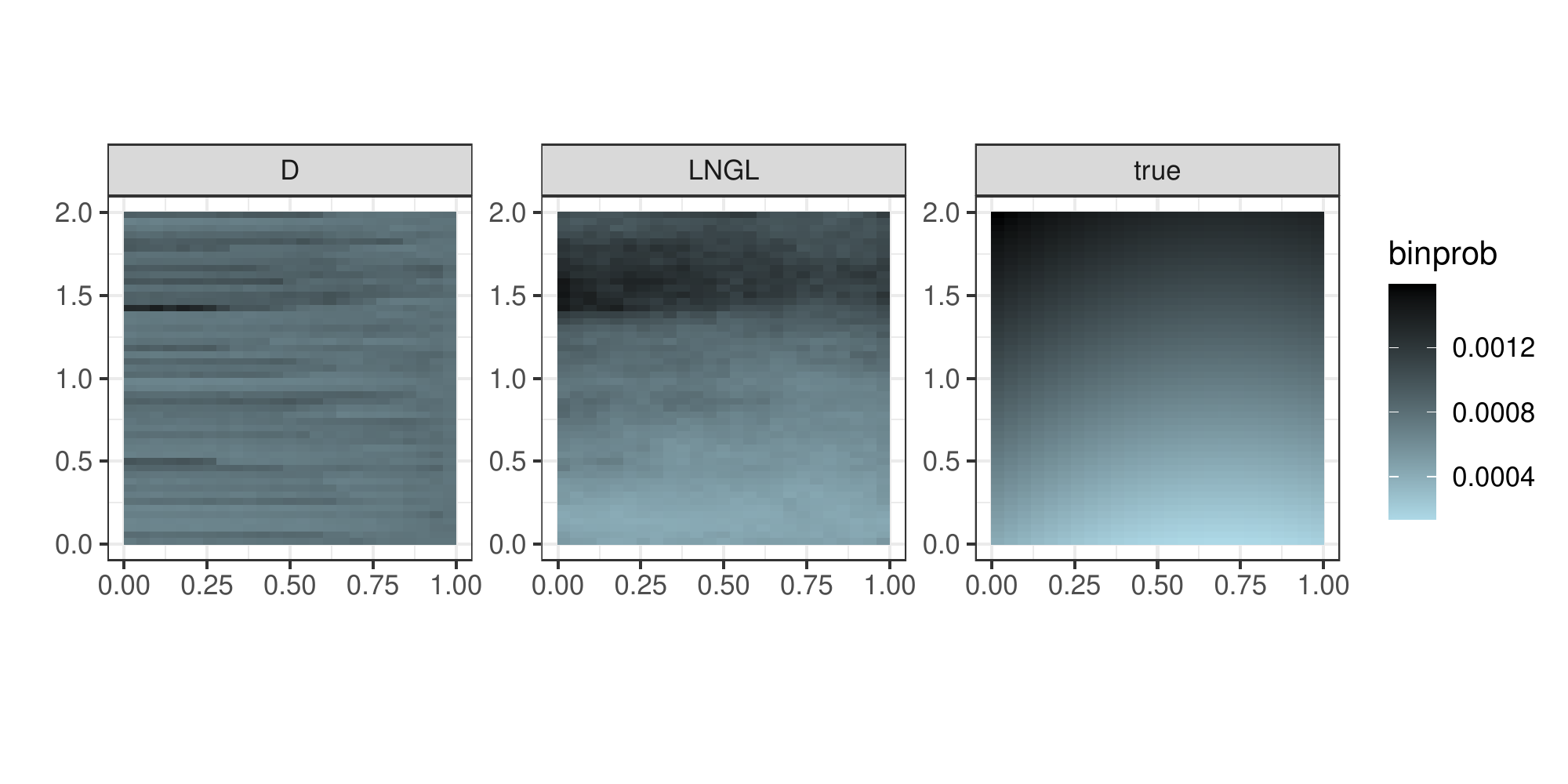}
	\end{center}
	\caption{Left and middle: posterior mean probabilities for D- and LNGL-priors respectively. Right: true posterior probabilities.  Horizontally $25$ bins, vertically $50$ bins. \label{fig:resB_postmean}}
\end{figure}
Smoothing induced by the LNGL-prior is quite apparent and the posterior mean estimate is visually more appealing. The latter is not surprising as the data-generating density is smooth.  

\subsubsection{Varying the number of bins}
In Figures \ref{fig:resB_postmean_coarser} and \ref{fig:resB_postmean_refined} we ran the algorithm on the same dataset using coarser/refined binning. 
\begin{figure}
	\begin{center}
		\includegraphics[scale=0.8]{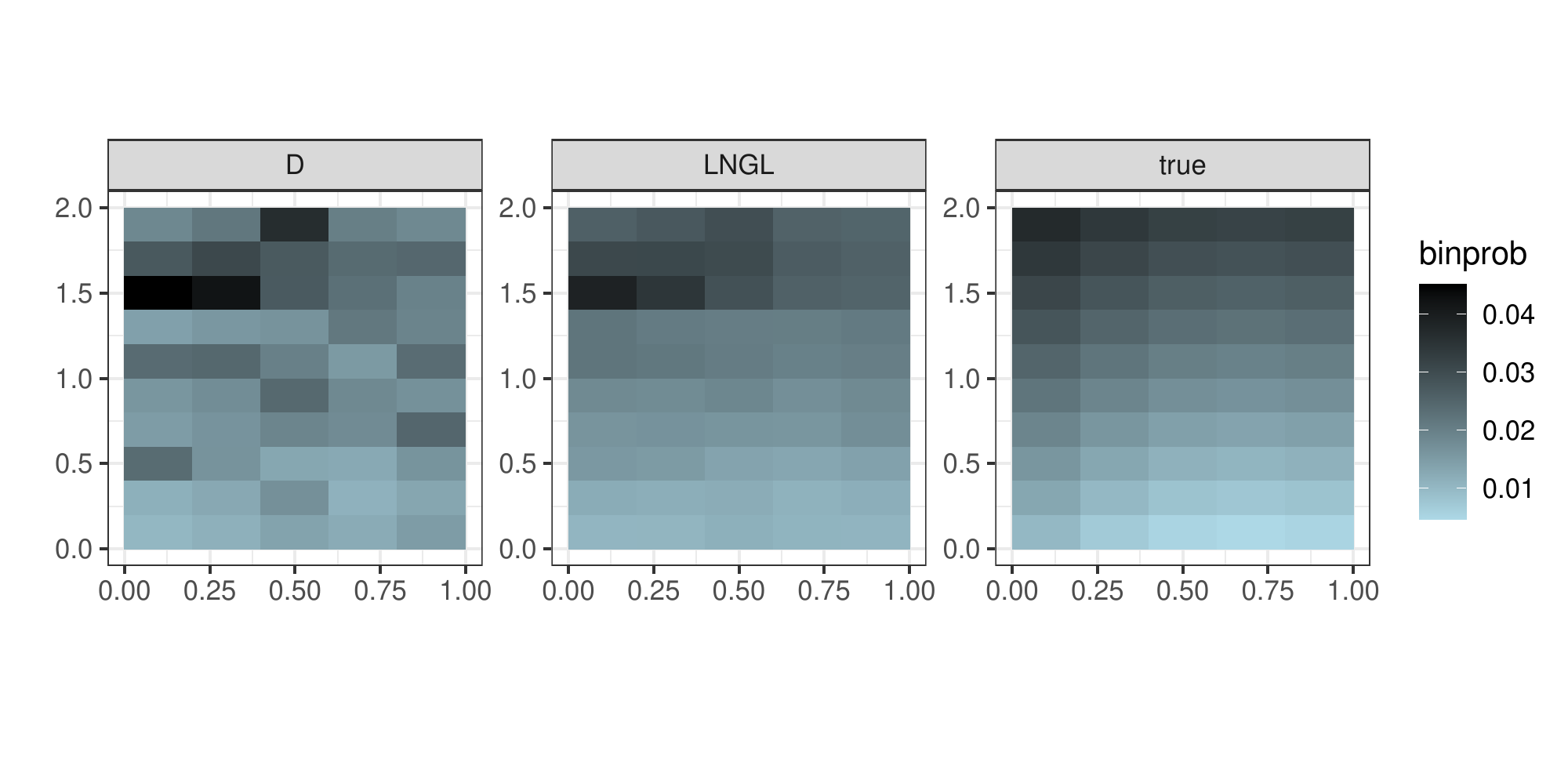}
	\end{center}
	\caption{Left and middle: posterior mean probabilities for D- and LNGL-priors respectively. Right: true posterior probabilities.  Horizontally $5$ bins, vertically $10$ bins. \label{fig:resB_postmean_coarser}}
\end{figure}

\begin{figure}
	\begin{center}
		\includegraphics[scale=0.8]{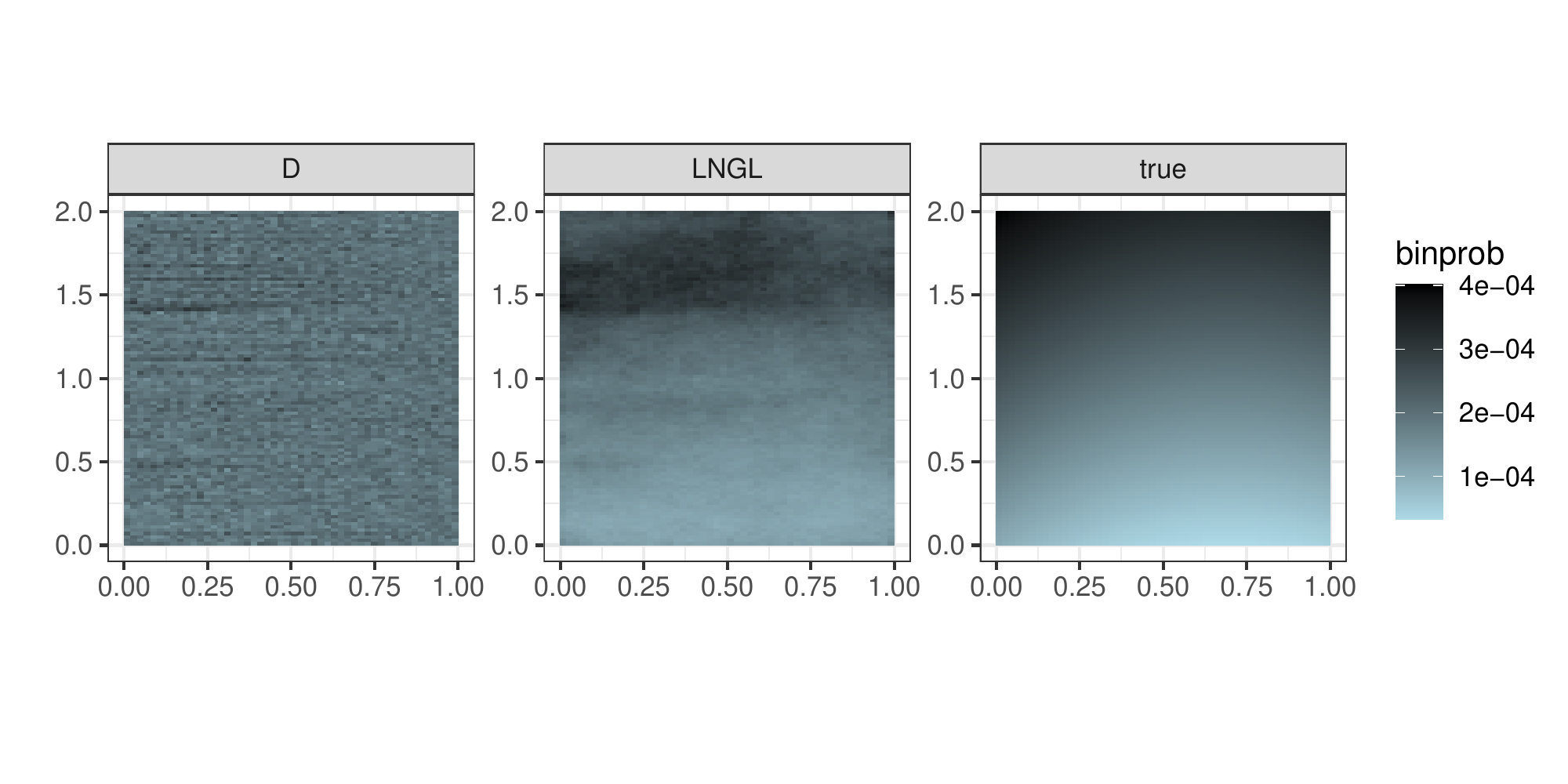}
	\end{center}
	\caption{Left and middle: posterior mean probabilities for D- and LNGL-priors respectively. Right: true posterior probabilities.  Horizontally $100$ bins, vertically $200$ bins. \label{fig:resB_postmean_refined}}
\end{figure}

For each combination of prior/binning, we compute the Wasserstein distance (based on the $\ell_1$/cityblock distance) between the probability vector of bin-masses of the estimate and the probability vector of bin masses of the true data-generating distribution.  For a  definition and motivation for using this distance we refer to  Chapter 2.1 in \cit{PanaretosZemel}. 
\begin{center}
\begin{tabular}{c || c| c| c}
prior/bins & $5/10$ & $25/50$ & $50/100$ \\ \hline\hline
D &   $0.118$   &    $0.224$       &   $0.239$         \\
LNGL &     $0.069$ &    $0.076$       &   $0.082$         \\ 
\end{tabular}
\end{center}
Here ``5/10'' for example means: 5 bins in the horizontal direction 10 bins in the vertical direction. 
Especially with a large number of bins, the performance of the posterior mean using the D-prior is visually unappealing. Inspection of the traceplot for the parameter $\tau$ in this setting reveals large uncertainty. The LNGL-prior on the other hand seems quite robust in performance, once a sufficiently large number of bins is chosen.

\subsubsection{Monte-Carlo study}

To compare the performance of the posterior mean estimator under both prior specifications we conduct a Monte-Carlo study. 
In each simulation run we
\begin{itemize}
	\item  simulate a dataset;
	\item compute the posterior mean under both the  D- and LNGL-prior;
	\item  compute the Wasserstein distance (based on the $\ell_1$/cityblock distance) between the probability vector of bin-masses of the estimate and the probability vector of bin masses of the true data-generating distribution.
\end{itemize}
We considered samples sizes
$100, 250, 500$ and took the Monte-Carlo sample size equal to $100$. In figures \ref{fig:resB_mc1} and \ref{fig:resB_mc2} we give two visualisations of the results. As expected, with large sample size the distance tends to be smaller. Additionally, the LNGL-prior appears to outperform the D-prior for this choice of data-generating distribution. 

\begin{figure}
	\begin{center}
		\includegraphics[scale=0.7]{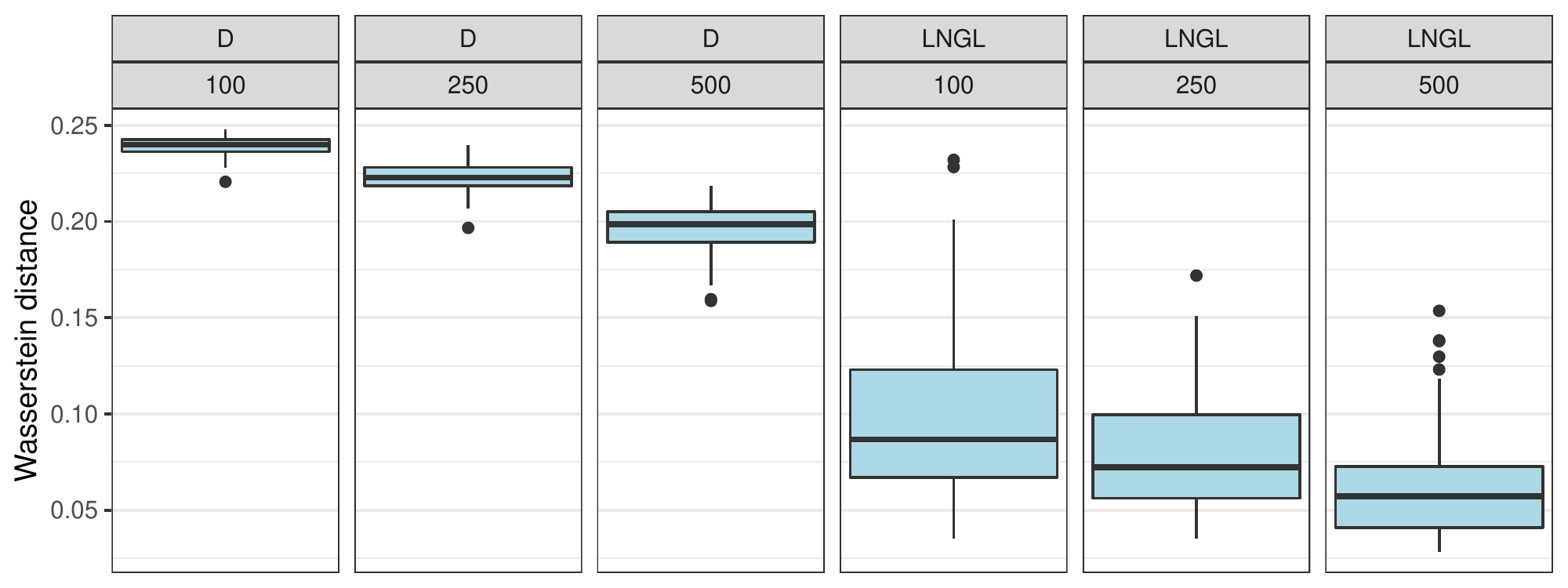}
	\end{center}
	\caption{Simulation study. Wasserstein distance, averaged over $100$ Monte-Carlo samples for samples sizes $100, 250, 500$ and both the D- and LNGL-prior. \label{fig:resB_mc1}}
\end{figure}

\begin{figure}
	\begin{center}
				\includegraphics[scale=1.0]{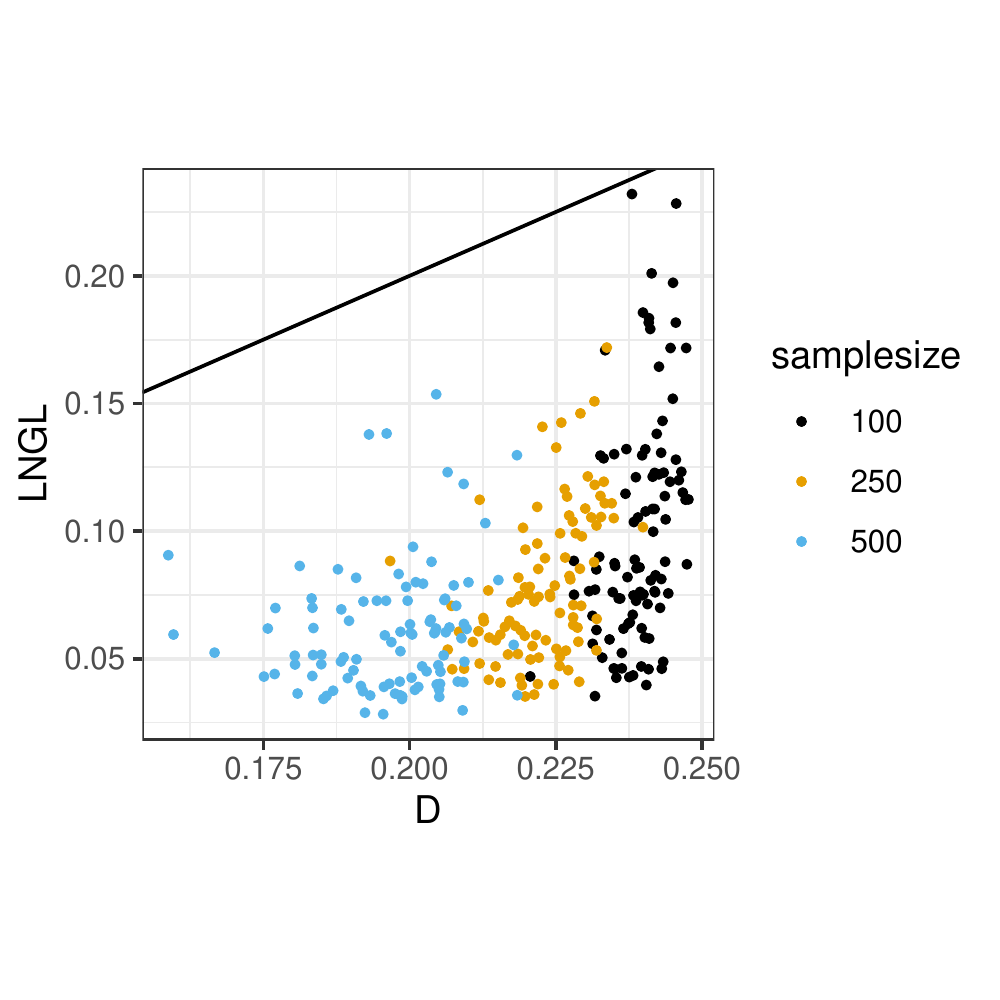}
	\end{center}
	\caption{Simulation study. Wasserstein distance for simulated datasets  for samples sizes $100, 250, 500$ and both the D- and LNGL-prior. Added is the line with intercept $0$ and slope $1$.  \label{fig:resB_mc2}}
\end{figure}

\section{Discussion}
The main theoretical contribution of this paper is Theorem \ref{thm:CSCMconsistency}.  We expect the  the rate in this theorem to be suboptimal. We  conjecture  $\eps_n$ to be the optimal rate, suggesting that  Lemma \ref{lem:KL_3} is sufficiently  sharp. However, due to the tests constructed in  Lemma \ref{lem:KL_3} we obtain  rate $\eps^{2/3}$ in our main result. We postpone further investigations to future research, where we also hope to obtain  rates in  global metrics.

The derived rate is non-adaptive. Adaptation can be achieved by employing a prior on the number of bins, see Chapter 10 in \cit{GhoVaart}. Computationally, such an approach is less attractive, as it requires a sampler on the larger space of the union of partitions. Instead, in our numerical work we have chosen to start off  from a larger number of small bins, and have a prior on the scaling parameter $\tau$ take care of the regularisation. This yields an easily implementable and efficient computational scheme.


\appendix
\section{Technical proofs}

\begin{lem}\label{lem:L1bound}
Let $f_1$ and $f_2$ be bivariate density functions on $\scr{M}=[0,M_1] \times [0,M_2]$. Assume the density of the censoring time, $g$, is bounded. Then there exists a constant $C>0$, independent of $f_1$ and $f_2$ such that 
\[	\|s_{f_1}-s_{f_2}||_1	\le C \|f_1-f_2\|_\infty. \]
\end{lem}
\begin{proof}
Say $g$ is bounded by $\bar{K}$. 
If $(t,z) \in \scr{M}$, then 
\begin{align*}
|s_{f_1}(t,z)-s_{f_2}(t,z)|&=\Bigg|g(t)\Bigg(\ind_{\{z>0\}}\int_0^t(f_1(u,z)-f_2(u,z))\dd u\\
&\quad+\ind_{\{z=0\}}\int_t^{M_1}\int_0^{M_2}(f_1(u,v)-f_2(u,v))\dd v\dd u\Bigg)\Bigg|\\
 & \le \bar{K} \max\left\{ \int_0^t \left|f_1(u,z)-f_2(u,z)\right|\dd u, \int_t^{M_1}\int_0^{M_2}\left|f_1(u,v)-f_2(u,v)\right|\dd v\dd u\right\} \\ & \le \bar{K}\, |\scr{M}|\,
\|f_1-f_2\|_\infty.
\end{align*}
This implies
\begin{equation}\label{eq:L1_h}\|s_{f_1}-s_{f_2}||_1=\int_{\mathcal{M}} |s_{f_1}-s_{f_2}|\dd\mu\le \bar{K}\, |\scr{M}|\, \mu(\scr{M})\, \|f_1-f_2\|_\infty.\end{equation}
\end{proof}

\begin{lem}\label{lem:subset_omega}
Impose Assumption \ref{CSCMasmp_f0}. 
Let $f_{0,n}$ be as defined in (\ref{eq:f0n_def}).
Define set \[
\Omega_n:=\left\{f\in\mathcal{F}: ||f-f_{0,n}||_\infty\le C\delta_n^{\rho},\, \mbox{supp}(f)\supseteq\mathcal{M}\right\}.\]
If $f \in \Omega_n$, then for sufficiently large $n$ there exists a constant $C_1>0$ such that 
\[
KL(s_{f_0},s_f)\le C_1\delta_n^{\rho},\quad V(s_{f_0},s_f)\le C_1\delta_n^{\rho}.\]
\end{lem}
\begin{proof}
The proof is based on Lemma B.2 in \cit{GhoVaart}, which gives the following inequalities
\begin{equation}
\label{eq:KL_V}
\begin{split}
KL(s_{f_0},s_f)&\le 2 h^2(s_{f_0},s_f)\|s_{f_0}/s_f\|_\infty,\\
V(s_{f_0},s_f)&\le 2 h^2(s_{f_0},s_f)\|s_{f_0}/s_f\|_\infty.
\end{split}
\end{equation}
Therefore, for $f\in \Omega_n$ we bound $h^2(s_{f_0}, s_f)$ and $||s_{f_0}/s_f||_\infty$. Substituting these bounds in \eqref{eq:KL_V} finishes the proof. 

 {\bf Step 1: showing that $h^2(s_{f_0}, s_f)\lesssim \delta_n^\rho$. }
By the definition of $f_{0,n}$ in (\ref{eq:f0n_def}), for any $(t,z)\in A_{n,j}\times B_{n,k}$,
\begin{align*}
|f_{0,n}(t,z)-f_0(t,z)|&=\left||A_{n,j}\times B_{n,k}|^{-1}\int_{A_{n,j}}\int_{B_{n,k}}f_0(u,v)dvdu-f_0(t,z)\right|\\
&\le|A_{n,j}\times B_{n,k}|^{-1}\int_{A_{n,j}}\int_{B_{n,k}}|f_0(u,v)-f_0(t,z)|\dd v\dd u\\
&\le \max_{(u,v)\in A_{n,j}\times B_{n,k}}|f_0(u,v)-f_0(t,z)|.
\end{align*}
By assumption (\ref{eq:f0cntns}) on $f_0$, we have
\[\max_{(u,v)\in A_{n,j}\times B_{n,k}}|f_0(u,v)-f_0(t,z)|\le c\max_{(u,v)\in A_{n,j}\times B_{n,k}}||(u,v)-(t,z)||^\rho\le L(2\sqrt{2}\delta_n)^\rho.\]
Hence
\begin{equation}\label{eq:f0f0n}
||f_{0,n}-f_0||_\infty =\max_{j,k} \left|\max_{(t,z)\in A_{n,j}\times B_{n,k}}|f_{0,n}(t,z)-f_0(t,z)| \right|\le  c(2\sqrt{2}\delta_n)^\rho.
\end{equation}
Applying Lemma \ref{lem:L1bound} with $f_1\equiv f_0$ and $f_2 \equiv f \in \Omega_n$ gives
\[||s_{f_0}-s_{f}||_1\le C(\|f_0-f_{0,n}\|_\infty+\|f_{0,n}-f\|_\infty)\lesssim \delta_n^{\rho}.\]
where we used \eqref{eq:f0f0n} to bound the first term on the right-hand-side. Using the inequality $h^2(f_1,f_2)\le \frac1{2}||f_1-f_2||_1$, we then have
\begin{equation}\label{eq:hellinger}h^2(s_{f_0},s_f)\le \frac1{2}||s_{f_0}-s_{f}||_1\lesssim \delta_n^{\rho}.\end{equation}

{\bf Step 2. showing that for sufficiently large $n$ there exists a constant $\tilde c>0$ such that    $||s_{f_0}/s_f||_\infty\le \tilde c$. }

First note that
\begin{equation}\label{eq:infty}
\left\|\frac{s_{f_0}}{s_f}\right\|_\infty\le\max\left\{\left\|\frac{\partial_2F_0}{\partial_2F}\right\|_\infty,\left\|\frac{1-F_{0,X}}{1-F_X}\right\|_\infty\right\}
\le \left\|\frac{f_0}{f}\right\|_\infty
\end{equation}
By Assumption \ref{CSCMasmp_f0}, there exists $\ul{M}$ such that $f_0(t,z)\ge \ul{M}$ for $(t,z) \in \scr{M}$. 
 Since $f\in \Omega_n$ we have $f(t,z)\ge f_{0,n}(t,z) - C\delta_n^\rho$. 
By Equation \eqref{eq:f0f0n}, there exists a $k$ (depending on $c$ and $\rho$) such that 
\[ |f_0(t,z) -f_{0,n}(t,z) | \le k \delta_n^\rho. \]
Therefore
\[  \frac{f_0(t,z)}{f(t,z)}  \le \frac{f_0(t,z)}{f_{0,n}(t,z)-C\delta_n^\rho} \le \frac{f_{0,n}(t,z) + k \delta_n^\rho}{f_{0,n}(t,z)-C\delta_n^\rho} \]
For $n$ sufficiently large we will have $C\delta_n^\rho \le \ul{M}/2$ and therefore the denominator of the right-hand-side can be lower bounded by $\ul{M}/2$. This implies boundedness of $\|f_0/f\|_\infty$  for $n$ sufficiently large.

\end{proof}

\section{Programming details in the Turing language}\label{sec:turing}
\label{sec:hmc}
For each observation, indexed by $i \in \{1,\ldots, n\}$, we compute the vector of indices corresponding to the shaded area in Figure \ref{fig:censoring-sketch}, as well as the fraction of the area that is shaded. This means that there is an $n$-dimensional vector {\tt ci}, where each element is of type ({\bf C}ensoring{\bf I}nformation). The structure {\bf C}ensoring{\bf I}nformation has two fields: {\tt ind} and {\tt fracarea}, holding indices and areafractions respectively. 
If {\tt L} is the graph-Laplacian (with a small multiple of the identity matrix added), the model is specified as follows:
\begin{lstlisting}
@model GraphLaplacianMod(ci,L) = begin
    tau ~ Exponential(1.0)
    H ~ MvNormalCanon(L/tau)
    Turing.@addlogprob! loglik(H, ci)
end
\end{lstlisting}
Here the loglikelihood is calculated using
\begin{lstlisting}
function loglik(H, ci)
    theta = softmax(H)
    ll = 0.0
    @inbounds for i in eachindex(ci)
        c = ci[i]
        ll += log(dot(theta[c.ind], c.fracarea))
    end
    ll
end
\end{lstlisting}

%

\bibliographystyle{plainnat}

\end{document}